\documentclass[a4paper]{amsart}

\usepackage{amsmath, amssymb, latexsym, amsfonts, pigpen, enumitem}
\usepackage[matrix,arrow,curve]{xy}

\numberwithin{equation}{section}
\theoremstyle{plain}
\newtheorem{theorem}[equation]{Theorem}
\newtheorem{corollary}[equation]{Corollary}
\newtheorem{lemma}[equation]{Lemma}

\newtheorem{proposition}[equation]{Proposition}

\theoremstyle{definition}
\newtheorem{definition}[equation]{Definition}

\theoremstyle{remark}
\newtheorem{remark}[equation]{Remark}

\newcommand{\pt}{{\rm pt}}
\newcommand{\GW}{{\mathrm{GW}}}
\newcommand{\W}{\mathrm{W}}
\newcommand{\K}{\mathrm{K}}
\newcommand{\id}{\operatorname{id}}
\newcommand{\Cone}{\operatorname{Cone}}
\newcommand{\chark}{\operatorname{char}}

\newcommand{\eff}{{\mathrm{eff}}}
\newcommand{\Dba}{\mathbf{D}^-}
\newcommand{\Dun}{\mathbf{D}}
\newcommand{\DMeff}{\mathbf{DM}_{\eff}^{-}}
\newcommand{\DM}{\mathbf{DM}}
\newcommand{\DMGW}{\mathbf{DM}^\mathrm{GW}}
\newcommand{\DMW}{\mathbf{DM}^\mathrm{W}}
\newcommand{\DMGWeff}{\mathbf{DM}^{\GW}_{\eff}}
\newcommand{\DMeffmGW}{\mathbf{DM}^{\GW}_{\eff,-}}
\newcommand{\DMWeff}{\mathbf{DM}^\mathrm{W}_{\eff}}
\newcommand{\DMeffmW}{\mathbf{DM}^{\W}_{\eff,-}}
\newcommand{\Daff}{\mathbf{D}^-_{\affl}}

\newcommand{\ShN}{Sh_{nis}}
\newcommand{\ShNGW}{\ShN(GWCor)}
\newcommand{\ShNW}{\ShN(WCor)}
\newcommand{\PreGW}{Pre(GWCor)}

\newcommand{\DbaShNGW}{\Dba(\ShNGW)}
\newcommand{\DbaShNW}{\Dba(\ShNW)}

\newcommand{\PreTr}{Pre(Cor_{virt})}
\newcommand{\PreTrs}{Pre^{tr}}
\newcommand{\ShNTr}{Sh_{nis}(Cor_{virt})}
\newcommand{\ShNTrs}{Sh^{tr}_{nis}}
\newcommand{\ShNTrss}{Sh^{tr}}

\newcommand{\DbaAffCorvs}{\Dba_{\affl}(\PreTrs)}

\newcommand{\DbaPs}{\DbaPreTrs}

\newcommand{\DbaSs}{\DbaShNTrs}
\newcommand{\iaffDbaP}{\Dba_{\affl-inv}(\PreTrs)}
\newcommand{\caffDbaP}{\Dba_{\affl-contr}(\PreTrs)}
\newcommand{\iaffDbaSh}{\Dba_{\affl-inv}(\ShNTrss)}
\newcommand{\caffDbaSh}{\Dba_{\affl-contr}(\ShNTrss)}

\newcommand{\Tau}{\mathcal T}

\newcommand{\bZtr}{\bZ_{tr}}

\newcommand{\cmod}{\text{-}\mathbf{mod}}

\DeclareMathOperator{\Supp }{Supp}
\DeclareMathOperator{\Ker}{Ker}

\DeclareMathOperator{\Coker}{Coker}

\DeclareMathOperator{\coker}{coker}

\newcommand{\bGm}{{\mathbb G_m}}

\newcommand{\affl}{\mathbb{A}^1}
\newcommand{\aff}{{\mathbb{A}}}
\newcommand{\bGmw}{{\mathbb G_m^{\wedge 1}}}

\newcommand{\bZ}{\mathbb Z}

\newcommand{\Cor}{Cor_{virt}}

\newcommand{\ShNis}{Sh_{nis}(Sm_k)}

\newcommand{\DbaPre}{\Dba(Pre)}
\newcommand{\DbaPreTr}{\Dba(\PreTr)}
\newcommand{\DbaPreTrs}{\Dba(\PreTrs)}
\newcommand{\DbaShNTr}{\Dba(\ShNTr)}
\newcommand{\DbaShNTrs}{\Dba(\ShNTrs)}
\newcommand{\DbaNisTr}{\Dba_{nis}(\PreTr)}
\newcommand{\DbaNisTrs}{\Dba_{nis}(\PreTrs)}

\newcommand{\HLoc}{Loc}

\newcommand{\FtrPre}{F_{tr}}
\newcommand{\FtrDPre}{F_{tr}}
\newcommand{\LtrPre}{L_{tr}}
\newcommand{\LtrDPre}{\mathcal L_{tr}}

\newcommand{\LtrDShN}{\mathcal L_{tr}^{nis}}
\newcommand{\FtrN}{F_{tr}^{nis}}
\newcommand{\lnistr}{l_{nis}^{\PreTrs}}
\newcommand{\CorHLoc}{\Cor(\HLoc)}
\newcommand{\CorHLocs}{\Cor^{\HLoc}}
\newcommand{\loc}{loc}
\newcommand{\loctr}{loc_{tr}}
\newcommand{\locD}{loc}
\newcommand{\locDtr}{loc_{tr}}
\newcommand{\PreHLoc}{Pre(\HLoc)}
\newcommand{\PreHLocs}{Pre^{\HLoc}}
\newcommand{\PreHLocTr}{Pre(\CorHLoc)}
\newcommand{\PreHLocTrs}{Pre^{\HLoc}_{tr}}

\newcommand{\DbaPreHLoc}{\Dba(\PreHLoc)}

\newcommand{\DbaPreHLocTrs}{\Dba(\PreHLocTrs)}

\newcommand{\LtrLoc}{L_{tr}^{loc}}
\newcommand{\LtrLocD}{\mathcal L_{tr}^{loc}}

\newcommand{\lnis}{l_{nis}} 
\newcommand{\lNisTr}{l_{nis}^{tr}} 

\newcommand{\LnisTr}{L_{nis}^{tr}}
\newcommand{\RnisTr}{R_{nis}^{tr}}

\newcommand{\LnisDTr}{L_{nis}^{tr}}

\newcommand{\DPreTrs}{\mathbf D(\PreTrs)}

\newcommand{\DMeffmCorv}{ \mathbf{DM}^{Cor_{virt}}_{ \mathrm{eff,-} } }
\newcommand{\DMeffCorv}{ \mathbf{DM}^{Cor_{virt}}_{ \mathrm{eff} } }
\newcommand{\DMeffmCorvr}{ \mathbf{DM}^{Cor_{virt},r}_{ \mathrm{eff,-} } }
\newcommand{\DMeffmCorvl}{ \mathbf{DM}^{Cor_{virt},l}_{ \mathrm{eff,-} } }

\newcommand{\bZtrNis}{\bZ_{tr,nis}}

\newcommand{\Dcontr}{\caffDbaP}

\newcommand{\preF}{F}
\newcommand{\bP}{P}

\begin{document}

\title{Effective Grothendieck-Witt motives of smooth varieties.} 


\author{Andrei Druzhinin}             

\email{andrei.druzh@gmail.com}

%
%
\address{ Chebyshev Laboratory, St. Petersburg State University, 14th Line V.O., 29B, Saint Petersburg 199178 Russia}








\begin{abstract}
The category of effective Grothendieck-Witt-motives $\DMeffmGW(k)$ is defined. 
The construction starts with the category of GW-correspondences over a field $k$. 
In the case of an infinite perfect base field $k$ such that $char\,k\neq 2$
using the Voevodsky-Suslin method
we compute the functor $M^{GW}_{eff}\colon Sm_k\to \DMeffmGW(k)$ of Grothendieck-Witt-motives of smooth varieties 
and prove
that 
for any smooth scheme $X$ and a homotopy invariant sheaf with GW-transfers $F$
$$
Hom_{\DMeffmGW(k)}(M^{GW}_{eff}(X),  F[i]) \simeq H^i_{nis}(X, F)
$$
naturally in $X$ and $F$.
\end{abstract}


\maketitle


%
%
%

\section{Introduction.}

\subsection{Categories of GW-motives and Witt-motives.}

In this article we construct the categories of effective GW-motives $\DMeffmGW(k)$ and Witt-motives $\DMeffmW(k)$
over an infinite perfect field $k$, $\chark k\neq 2$. 
The construction is done by the Voevodsky-Suslin method originally used for the construction of the triangulated category of motives $\mathbf{DM}(k)$ and effective motives $\mathbf{DM}^{-}_\mathrm{eff}(k)$ (\cite{VSF_CTMht_Ctpretr}, \cite{VSF_CTMht_DM}, \cite{MVW_LectMotCohom}).

According to the Voevodsky-Suslin method the starting point is some additive category of correspondences between smooth schemes. In the case of GW-motives it is the category of so-called Grothendieck-Witt correspondences $GWCor_k$ and 
for the case of Witt-motives it is the category of Witt-correspondences $WCor_k$. Morphisms in these categories are defined by classes of quadratic spaces in the same category as the one used in the definition of $\K_0$-correspondences 
studied by Walker in \cite{W_MotComK} and by Garkusha and Panin in \cite{GarPan-Kmot12} and \cite{GarPan-Kmot14}. 
Namely, we consider 
the full subcategory $\mathcal P(X,Y)$ of the category of coherent sheaves on $X\times Y$ spanned by  
such sheaves $P$ that $\Supp P\subset X\times Y$ is finite over $X$ and the direct image of $P$ on $X$ is locally free.
To define the notion of quadratic space we consider the category $\mathcal P(X,Y)$ 
with the duality functor defined by $\mathcal Hom (-,\mathcal O(X))$. 

If we have the category of effective motive $\DMGWeff(k)$ (or $\DMWeff(k)$) 
equipped with a tensor structure, 
then we can define the category of non-effective motives $\DMGW(k)$ ($\DMW(k)$) 
as the homotopy categories of $\bGmw$-spectra in $\DMGWeff(k)$ ($\DMWeff(k)$). 

The category $\DMGW(k)$ 
gives a version of generalised motivic cohomology theory defined by presheaves $X\mapsto Hom_{\DMGW}(M^{GW}(X), \mathbb Z_{GWtr}(i)[j]))\simeq H^j_{nis}(X,\mathbb Z_{GWtr}(i))$ for $i\geq 0$, where $\mathbb Z_{GWtr}(i) = M^{GW}(\mathbb G_m^{\wedge i})[i]$. 
It is written in \cite{FB_EffSpMotCT} that 
if we consider the cohomology theory defined above with $\mathbb Z[1/p]$ coefficients, $p=\chark k$, then it is
isomorphic to 
the generalised motivic cohomology defined by Calmes and Fasel in \cite{CF_FinChWittCor}
with the same coefficients.
And so the result of \cite{GG_RecRatStMot} states that both of this categories 
are rationally isomorphic to the motivic homotopy groups, 
and corresponding categories of motives are rationally equivalent to 
$\mathbf{SH}(k)_{\mathbb Q}$.
Under this equivalence the computation of the GW-motives of smooth scheme $X$ presented here gives 
model for $\Omega_{\mathbb G_m}^\infty\Sigma_{\mathbb G_m}^\infty(X)$
in $\mathbf{SH}(k)_{\mathbb Q}$.
It is important to note that the functor $\mathbf{SH}(k)\to \DM^{GW}(k)$ can be constructed explicitly applying the reconstruction of $\mathbf{SH}(k)$ as the category of framed motives \cite{GarkushaPanin_FrMot} 

In \cite{ALP_WittSh-etsinvert} 
Ananievsky, Levine, Panin 
have constructed
the category of Witt-motives via the category of modules over the Witt-ring sheaf, 
and it is proved that this category satisfies Morel's conjecture, namely, 
it is such triangulated category that 
is rationally equivalent to the minus part $\mathbf{SH}^-(k)_{\mathbb Q}$ of the stable motivic homotopy category.
Hypothetically constructed here category $\DMW(k)$ is equivalent to the category of Witt-motives from \cite{ALP_WittSh-etsinvert}, and so the result of this article gives a fibrant replacements in this category. 

\subsection{The main result.}

According to the universal property of a category of motives, 
any category of motives satisfying Nisnevich Mayer-Vietoris property and homotopy invariance property should be defined as a localisation of some triangulated category with respect to the Nisnevich squares and morphisms of the type $X\times\affl\to X$. 
In case of effective $\GW$-motives, this triangulated category is 
the derived category of presheaves with $\GW$-transfers $\Dba(\PreGW)$. 
The main advantage of the Voevodsky-Suslin method is that it gives more explicit description of the category of effective motives as the full subcategory of the derived category $\Dba(\ShNGW)$ and provides a computation of the motives of smooth varieties. 

The main results we've proved in the case of GW-(Witt-)motives have almost the same formulation to the original case of $\DMeff$ 
except that we should consider sheafification of the presheaves $GWCor(-,X)$ in the formula for motive of a smooth scheme, since these presheaves aren't sheaves. 
\begin{definition}\label{def:intr:DMeff}
The \emph{category of effective GW-motives} $\DMeffmGW(k)$ (\emph{or Witt-motives} $\DMeffmW(k)$) 
over an infinite perfect field $k$, $\chark k\neq 2$, is 
the full subcategory of the derived category $\DbaShNGW$ ($\DbaShNW$)
spanned by complexes $A^\bullet$ with homotopy invariant sheaf cohomologies $h_{nis}^i(A^\bullet)$.
The functor $M^{GW}\colon Sm_k\to \DMeffmGW(k)$ 
is defined as
\begin{multline}\label{eq:intr:MGW}
M^{GW}(X) = \mathcal Hom_{\DbaShNGW}(\Delta^\bullet, {GWCor}_{nis}(-,X) )=\\
[\cdots \to {GWCor}_{nis}(-\times\Delta^i,X)\to \cdots
 \cdots \to {GWCor}_{nis}(-,X)]
,\end{multline}
where $\Delta^i$ denotes affine simplexes (see def. \ref{def:affsimpl}),
${GWCor}_{nis}(-,X)$ denotes Nisnevich sheafification of 
$GWCor(-,X)$ 
equipped with GW-transfers in a unique possible way (see proposition \ref{prop:uniqnistrans}). 
\end{definition}
\begin{theorem}\label{th:intr:DMeff} 
Suppose $k$ is an infinite perfect field, $\chark k\neq 2$;
then \begin{itemize}[leftmargin=10pt]
\item[a)](def. \ref{def:DMeff} and th. \ref{th:DMeff} points 1 and 2) the category $\DMeffmGW(k)$
is equivalent to 
the localization of the derived category $\DbaShNGW$ 
in respect to morphisms of the form $X\times \affl\to X$, $X\in Sm_k$,
and the localization functor is isomorphic to 
\begin{equation}\label{eq:intr:SCstShN}
C^*(-)=\mathcal Hom(\Delta^\bullet,-)\colon \DbaShNGW\to \DMeffmGW(k).
\end{equation}
\item[b)](theorem \ref{th:DMeff} point 3) there is a natural isomorphism
\begin{equation}\label{eq:intr:MTC}
Hom_{\DMGWeff(k)}(M^{GW}_{eff}(X),\mathcal F[i]) \simeq H^i_{nis}(X,\mathcal F), i\geq 0,
\end{equation}
for any
smooth variety $X$ and
homotopy invariant sheaves with GW-transfers $\mathcal F$.
\item[c)](proposition \ref{prop:MotiveProd})
There is a tensor structure on 
$\DMeffmGW (k)$ such that
\begin{equation}\label{eq:intr:tensprodDMGW}M^{GW}_{eff}(X)\otimes M^{GW}_{eff}(Y) \simeq M^{GW}_{eff}(X\times Y) \end{equation}
for $X,Y\in Sm_k$.
\end{itemize}
\end{theorem}

Then one can define the category of (non-effective) GW-motives and Witt-motives by stabilisation in respect to $\bGmw=\bGm/pt$ and
using the cancellation theorem for GW-correspondences (Witt-correspondences)
deduce the isomorphism \eqref{eq:intr:MTC} for (non-effective) motives,
though here we are concentrated on the category of effective motives and 
we
leave the cancellation theorem for further works. 

\subsection{About the method.}

Let's present 
our  result in the following abstract form 
\begin{theorem}\label{th:intr:abs}
Suppose $\Cor$ is a category enriched over abelian groups and such that objects of $\Cor$ are smooth schemes over a base $S$,
and there is a functor $Sm_S\to \Cor$, which is identity on objects 
($\Cor$ is a ringoid in sense of \cite[def. 2.1]{GarPan-Kmot14}).
Suppose in addition that for any essential smooth local henselian scheme $U$ over $S$ the relations (SHI) and (NL)
holds (see the next page for (SHI) and (NL)).

Let $\DMeffmCorv(S)$ (and $\DMeffCorv(S)$) be 
the localisation of $\DbaPreTrs$ (and $\DPreTrs$)
with respect to Nisnevich quasi-isomorphisms
(homomorphisms of complexes of presheaves such that induced homomorphisms of germs are quasi-isomorphisms, see def. \ref{def:Dnisloc}) 
and morphisms of the form $X\times\affl\to X$, $X\in Sm_S$.
Then 
$\DMeffmCorv(S)$ ($\DMeffCorv(S)$) satisfies 
the similar properties to the points a,b,c of theorem \ref{th:intr:DMeff}.
\end{theorem}

The condition (SHI) 
is equivalent to the Strictly Homotopy Invariance theorem, which is proven for $\GW$- and Witt-correspondences in \cite{AD_StrHomInv}. 
All what we need except (SHI) 
is condition (NL), and the condition  
(NL) which is very easy to check.
The assumptions (NL) and (SHI) are equivalent to the definition of strictly $\affl$-invariant 
V-ringoid without the second point \cite[def. 2.4]{GarPan-Kmot14}. 
In fact, 
we reduce both conditions from the definition of V-ringoid 
that deal with topology 
to one condition (NL). 
The condition (NL) uses the correspondences between local schemes which are defined formally as morphisms in the category of pro-objects of the category $\Cor$, and for all known examples such correspondences can be defined explicitly.
It is easy to see that the condition (NL) 
holds for any category of correspondences $\Cor$ with "finite supports", i.e. 
such a ringoid $\Cor$ that any 
$\Phi\in \Cor(X,Y)$ can be represented as a composition $X\xrightarrow{\Phi^\prime} Z\xrightarrow{g} Y$, where $Z$ is finite scheme over $X$, $\Phi^\prime\in \Cor(X,Z)$, and $g$ is a morphism of schemes; so (NL) is true for all known examples of ringoids studied by the Suslin-Voevodsky method.
Moreover if we assume (SHI) and properties a),b) from the theorem, then (NL) follows. Thus we
get an easy criteria when the Voevodsky-Suslin method is applicable
\footnote{It was explained to the author by G.~Garkusha that there is another axiom instead of (NL) which can be used to replace all axioms dealing with Nisnevich topology and this axiom doesn't involve correspondences between local schemes. So in different cases of $\Cor$ it can be easier to check different variants of this axiom.}. 

To prove theorem \ref{th:intr:abs} we consider the following commutative diagram of adjunctions
\footnote{The third square relates to the cancellation theorem that is not subject of the article.} 
\footnote{In the case of bounded above derived categories there are only continuous (undotted) arrows of the diagram and the dotted arrows exist in the case of unbounded derived categories.}
\begin{equation}\label{diag:Constr}\xymatrix{
\Dun(Pre) \ar[r]^{L_{tr}}\ar[d]& \DbaPreTrs \ar[r]^{L_{\affl}^{tr}}\ar[d] \ar@<0.5ex>[l]^{F_{tr}}& 
 \DbaAffCorvs \ar[r]^(0.4){\Sigma^{\infty}_{\bGm}}\ar[d] \ar@<0.5ex>[l]& 
  \mathrm{Ho}(\mathrm{Sp}_{\bGmw}\DbaAffCorvs)\ar[d]\ar@<0.5ex>@{-->}[l]\\
\Dba(\ShN) \ar[r]^{L_{tr}^{nis}}\ar@{}[ru]|{(1)} \ar@<1ex>@{-->}[u]& 
 \DbaShNTrs \ar[r]^{L^{nis,tr}_{\affl}}\ar@<1ex>@{-->}[u]\ar@{}[ru]|{(2)} \ar@<0.5ex>[l]& 
  \DMeffmCorv \ar[r]|{\Sigma^{\infty}_{\bGm}}\ar@<1ex>@{-->}[u]\ar@{}[ru]|{(3)} \ar@<0.5ex>[l]& 
   \mathbf{DM}^{\Cor}\ar@<1ex>@{-->}[u]\ar@<0.5ex>@{-->}[l],\\ 
}\end{equation}
where 
$\PreTrs=\PreTr$, $\ShNTrs=\ShNTr$,
$F_{tr}$ is forgetful, $L_{tr}$ is the left derived to the left Kan extension, and $L^{tr}_{\affl}$ is the localisation in respect to morphisms $X\times\affl\to X$. 
The first row exists for an arbitrary ringoid $\Cor$ by general arguments, and the category $\DbaAffCorvs$ always satisfies the similar properties as in def. \ref{def:DMeff} and th. \ref{th:DMeff} (with presheaves instead of sheaves).
The second row is obtained by Nisnevich localisation form the first and the commutativity of the diagram yields the required properties of $\DMeffmCorv$.
%
%

The existence of the second row and the commutativity of diagram \eqref{diag:Constr} is equivalent to the exactness of the functors in the first row in respect to Nisnevich quasi-isomorphisms. The right adjoint functors in the first row are Nisnevich exact for any $\Cor$ by a trivial reasons.
The Nisnevich exactness of the left adjoint functors follows from
a good correlation of the Nisnevich topology with the correspondences $\Cor$ and with the structure of the category with interval on $Sm_k$: 
\par
1) 
$L_{tr}$ is Nisnevich exact because of 
the equalities 
$$\label{eq:intr:NisTr}
\Cor(U,X)=\bigoplus\limits_{x\in X} \Cor(U,X^h_x),
\eqno (NL)$$
where $U$ is local henselian and $X\in Sm_k$.

2) 
$L^{tr}_{\affl}$ is Nisnevich exact 
because of 
the equalities 
$$
\label{eq:intr:NisAffTr}
H^i_{nis}(\affl\times U,\mathcal F)=\begin{cases}\mathcal F(\affl\times U)=\mathcal F(U), i=0,\\ 0, \text{otherwise} ,\end{cases}
\eqno (SHI)$$
where $U$ is local henselian and $\mathcal F$ is a homotopy invariant presheaf with transfers; the equality \eqref{eq:intr:NisAffTr} is equivalent to the strictly homotopy invariance theorem proved in \cite{AD_StrHomInv}.

3) The commutativity of the third square follows from the equalities
\begin{equation*}\label{eq:intr:NisGmTr}
H^i_{nis}(\bGm\times U,\mathcal F)=\begin{cases}\mathcal F(\bGm\times U), i=0,\\ 0, \text{otherwise} \end{cases}
\end{equation*}
for the same $U$ and $\mathcal F$ as above; the last equality follows from the same excision isomorphisms proved in \cite{AD_StrHomInv}
on which the proof of the strictly homotopy invariance theorem is based.

Let's note 
that according to the standard variant of the Voevodsky-Suslin method to prove that the category $\ShNTrs$ is abelian and that
\begin{equation}\label{eq:intr:CohNisTr=CohNis}Ext_{\ShNTr}(\mathbb Z_{tr,nis}(X),\mathcal F)\simeq H^i(X,\mathcal F)\end{equation}
we should present a construction of transfers on the Nisnevich sheafification and contracting homotopy for the complex
$\dots\to\bZ_{tr}(\mathcal U^n)\to \dots\to \bZ_{tr}(\mathcal U)\to \bZ_{tr}(X)$ 
for any $X\in Sm_k$ and Nisnevich covering $\mathcal U\to X$.
Indeed at least in the case of $\GW$- (Witt-) correspondences such construction exists, see \cite[remark 3.16]{AD_DMGWeff} for the construction of transfers on the sheafification, but even in this particular case this way is much longer and more technical.
Let's note also that there is a way to proof the theorem for $\GW$- and Witt-correspondences via the method of \cite{GarPan-Kmot14}, since $\GW$ and $Witt$-correspondences are strictly $\affl$-invariant V-ringoid.
The main feature of 
the proof presented here is that it 
allows 
to separate general arguments and special properties of correspondences and 
to concentrate multiple checks relating to Nisnevich topology in one easily checked criteria. 

\subsection{The text review.}

In section \ref{sect:correspondences} we recall the definition of GW- and Witt-correspondences and prove some elementary properties used in the next sections, in particular we prove relation (NL). 

Sections \ref{sect:GWNiscoh} deals with the first square of \eqref{diag:Constr}. 
It is proven that 
for any correspondences satisfying condition (NL) 
category $\ShNTrs$ is abelian and
isomorphism \eqref{eq:intr:CohNisTr=CohNis} holds. 
Let's note that formally 
we should prove that $\ShNTrs$ is abelian before \eqref{eq:intr:CohNisTr=CohNis}, 
because otherwise the terms of \eqref{eq:intr:CohNisTr=CohNis} are undefined.
Nevertheless, in section \ref{sect:GWNiscoh} 
we prove firstly equality \eqref{eq:intr:CohNisTr=CohNis} in some modified form 
that has sense in any case (see lemma \ref{lm:HomNisloc=CohNis} and def. \ref{def:Dnisloc}) 
and then we deduce both mentioned results simultaneously.

In section \ref{sect:TensStr} we define the category $\DMeffmCorv$ as $\affl$-localisation of 
$\DbaShNTrs$ and define the tensor structure on the category of motives.

Section \ref{sect:DMeff} deals with the second square of \eqref{diag:Constr}. 
We prove that "action" of affine line as an interval in $Sm_k$ induce a semi-orthogonal decomposition on the bounded above derived category $\DbaShNTrs=\langle\iaffDbaSh,\caffDbaSh\rangle$,
where 
$\iaffDbaSh$ is spanned by homotopy invariant objects, and 
$\caffDbaSh$ are generated by $\affl$-contractable objects.
So the category of effective GW-motives $\DMeffmCorv$ is "homotopy invariant part" of $\DbaShNTrs$. 
This yields that the category of effective motives defined in the previous section is equivalent to the one given by def. \ref{def:DMeff} and theorem \ref{th:intr:abs}.

\vspace{3pt}

\subsection{Acknowledgements}
Acknowledgement to I.~Panin, who encouraged me to work on this project, for helpful discussions.

\subsection{Notation}
All considered schemes are a separated noetherian schemes of finite type over a separated noetherian base unless otherwise noted, 
and $Sm_k$ is category of smooth schemes over a field $k$.
We denote $Coh(X)=Coh_X$ the category of coherent sheaves on a scheme $X$.
%
We denote by $Z(f)$ the vanish locus of a regular function $f$ on a scheme $X$.
%
We use the symbols $\mathcal L$ and $\mathcal R$ for the left and right derived functors.
And let's note that we always consider objects of an abelian category as an objects in the derived category via the complexes concentrated in degree zero.

\section{GW- and Witt-correspondences}\label{sect:correspondences}

In this section we recall definition of GW- and Witt-correspondences and prove some properties used in next sections. 
All definitions are given over an separated noetherian base scheme $S$. 

\begin{definition}\label{def:catCohfcalP}
For a morphism of schemes $p\colon Y\to X$
let's denote by $Coh_{fin}(p)$ (or $Coh_{fin}(Y_X)$)
the full subcategory of the category of coherent sheaves on $Y$ 
  spanned by sheaves $\mathcal F$ such that $\Supp \mathcal F$ is finite over $X$.
Denote by  
$\mathcal P(p)$ (or by $\mathcal P(Y_X)$) 
the full subcategory of $Coh_{fin}(Y)$ 
  spanned by sheaves $\mathcal F$ such that $p_*(\mathcal F)$ is locally free sheaf on $X$. 

For schemes $X$ and $Y$ over a base scheme $S$ let 
$$Coh_{fin}^S(X,Y) = Coh_{fin}(X\times_S Y\to X),\;\mathcal P^S(X,Y) = \mathcal P(X\times_S Y\to X).$$
Sometimes we omit the index $S$ if it is clear from the context. 
\end{definition}
\begin{remark}
In the case of morphism of affine schemes $Y\to X$,
$\mathcal P(Y\to X)$ is equivalent to the full subcategory in the 
category of $k[Y]$-modules with objects being finitely generated and projective over $k[X]$.
\end{remark}

The functors $k[Y]\cmod\to {k[Y]\cmod}^{op}\colon M\mapsto Hom_{k[X]}(M,k[X])$ for finite morphisms $Y\to X$ of affine schemes,
defines in a canonical way a functor $$D_X\colon Coh_{fin}(Y_X)\to Coh_{fin}(Y_X)^{op}$$ 
for any morphism of schemes $Y\to X$.
Moreover there for any morphism $p\colon Y\to X$ is a natural isomorphism $D_X^2\simeq Id_{Coh_{fin}(p)}$ and the restriction of $D_X$ to the subcategory $\mathcal P(p)$ gives an exact category with duality $(\mathcal P(Y\to X),D_X)$
(see \cite[section 2]{AD_StrHomInv} for details).

The tensor product of the coherent sheaves induce the functors of the categories with duality
\begin{equation}\label{eq:compCorCor}-\circ -\colon  (\mathcal P^S(Y,Z),D_Y) \times (\mathcal P^S(X,Y),D_X) \to (\mathcal P^S(X,Z),D_X)\end{equation}
defined for any schemes $X$, $Y$ and $Z$ over the base $S$ and natural in $X,Y,Z$.

\begin{definition}\label{def:WCor}
\emph{The category $GWCor_S$} is the additive category  
with objects smooth schemes over $S$
and homomorphism groups defined as 
$$GWCor_S(X,Y) = GW(\mathcal P^S(X,Y),D_X)$$
where $GW$ is Grothendieck-Witt-group of the exact category with duality,
i.e. the group completion of the groupoid of non-degenerate quadratic spaces $(P,q)\colon $ $P\in \mathcal P^S(X,Y)$, $q\colon P\simeq D_X(P)$.

The composition in $GWCor$ is induced by functor \eqref{eq:compCorCor},
and identity morphism $Id_X=[(\mathcal O(\Delta),1)],$ 
where $\Delta$ denotes diagonal in $X\times_S X$.

In the same way we define \emph{the category $WCor_S$} via the Witt-groups of the exact category with duality, which are the factor-groups of $GW$ such that classes of metabolic spaces are zero
(see Balmer, \cite{Bal_DerWitt}). 
 
\end{definition}  
There is a functor $Sm_k\to GWCor$ ($Sm_k\to WCor$)
that takes a morphism $f\in Mor_{Sm_k}(X,Y)$ to the class of the quadratic space
$[(\mathcal O(\Gamma_f), 1)],$
where $\Gamma_f\subset Y\times X$ denotes the graph of a morphism $f\in Mor_{Sm_k}(X,Y)$, 
and $1$ denotes the unit quadratic form on a free coherent sheaf of a rank one on $\Gamma_f\simeq X$, 

As usual we call by \emph{a presheaf of abelian groups} on $Sm_S$ an additive functor $F\colon Sm_S\to Ab$, i.e. a functor such that $F(X_1\amalg X_2)=F(X_1)\oplus F(X_2)$. A \emph{presheaf with GW-transfers} is an additive functor $F\colon GWCor_S\to Ab$,
and similarly presheaves with Witt-transfers is an additive functor $F\colon WCor_S\to Ab$. 
A \emph{homotopy invariant presheaf} on $Sm_S$ is a presheaf $\mathcal F$ such that the natural homomorphism $\mathcal F(X)\simeq \mathcal F(\affl\times X)$ is an isomorphism.
A homotopy invariant presheaf with GW-(Witt-)transfers is a presheaf with GW-(Witt-)transfers that is homotopy invariant as a presheaf on $Sm_S$.     

The homotopy invariant presheaves with GW-transfers and Witt-transfers satisfy following important properties:
\begin{theorem}[see \cite{ChepInjLocHIiWtrPreSh} theorem 1, for $WCor$; see \cite{AD_StrHomInv} theorem 7.3, for $GWCor$]\label{th:HGWCorinj} 
Suppose $\mathcal F$ is a homotopy invariant presheaf with GW-transfers (Witt-transfers)  
and $U$ is a local scheme over a base filed $k$,
then the restriction homomorphism 
$\mathcal F(U)\to \mathcal F(\eta)$
is injective, where $\eta\in U$ is generic point.
\end{theorem}

\begin{theorem}[see \cite{AD_WtrSh} theorem 3, for $\mathcal F_{nis}$ in the case of $WCor$; see \cite{AD_StrHomInv} theorem 8.3, theorem 8.5, for $h_{nis}(\mathcal F_{nis})$, $i\geq 0$ 
in both cases $GWCor$, $WCor$]
\label{th:StrictHGWCor}
Suppose $\mathcal F$ is a homotopy invariant presheaf with GW-transfers (Witt-transfers)
over an infinite perfect field $k$, $char k\neq 2$,
then the associated Nisnevich sheaf ${\mathcal F}_{nis}$ and Nisnevich cohomology presheaves $h_{nis}({\mathcal F}_{nis})$ are homotopy invariant.
 
\end{theorem}

Finally we prove the equality \eqref{eq:intr:NisTr} for GW- and Witt-correspondences, which is used in the next section.
Let's recall the following definition, \cite[def. 2.1]{GarPan-Kmot14}:

\begin{definition} \label{def:Corvirt}
An additive category $\Cor$ with objects being smooth schemes and equipped with a functor $Sm_k\to \Cor$, which is identity on objects, is called a \emph{ringoid}. 
\end{definition}
Let's rewrite equation \eqref{eq:intr:NisTr} in the following abstract form
\begin{equation}\label{eq:CorvNisLoc} \Cor(U,X)=\bigoplus_{x\in X} \Cor(U,X^h_x)\end{equation}
for a ringoid $\Cor$, essentially smooth local henselian $U$, and smooth $X$.

\begin{definition}\label{def:Fin}
Given a morphism of schemes $p\colon Y\to X$, let's denote 
by $Fin(p)$ 
a filtering ordered set of closed subschemes in $Y$ that are finite over $X$.

For any two schemes $X,Y$ over some base scheme $S$,
let $Fin_S(X,Y)=Fin(Y\times_S X\to X)$, $Fin_S(X)=Fin(X\to S)$. 

For any closed subscheme $Z^\prime\subset X$ and a local scheme $U$ denote $Fin^Z(U,X)$ the set of closed subschemes 
$Z\in Fin_S(X)$ 
such that
$red(Z\times_U u)\subset Z^\prime$.

\end{definition}\begin{lemma}\label{lm:ProjCfinlim}
For any morphism of schemes $Y\to X$
\begin{equation*}\label{eq:WCorfininjlim}
Coh_{fin}(Y\to X) \simeq \varinjlim\limits_{Z\in Fin(Y\to X)} Coh(Z),\quad
\mathcal P(Y\to X)\simeq \varinjlim\limits_{Z\in Fin(Y\to X)}\mathcal P(Z\to X)
,\end{equation*}
and the equivalences are compatible with the duality functors $D_X$, so the second equivalence is equivalence of categories with duality.
\end{lemma}

\begin{lemma}\label{lm:locHendProjCsplit}
For any scheme $Z$ finite over a local henselian scheme $U$
we have
$$\mathcal P^S(Z\to U) = \prod\limits_{z\in Z}\mathcal P^S(Z^h_z\to U),$$
and this decomposition is compatible with the action of endofunctors $D_U$.
\end{lemma}
\begin{proof}
The claim follows from lemma \ref{lm:ProjCfinlim} and from that any finite scheme over local henselian scheme splits into disjoint union of local henselian schemes.
\end{proof}

\begin{proposition} \label{prop:GWWittNisloc}
The categories $GWCor_S$ and $WCor_S$ satisfy condition \eqref{eq:CorvNisLoc} for any base scheme $S$, and a local henselian $U$, and a scheme $X$ over $S$.
\end{proposition}
\begin{proof}
Indeed, the equality holds at the level of the categories $(\mathcal P(U,X),\mathcal D_X)$.
Let $u$ be the closed point of $U$,
then using equivalences from lemmas \ref{lm:ProjCfinlim}, \ref{lm:locHendProjCsplit} and notation above (def . \ref{def:Fin}) 
we get
\begin{multline*}
\mathcal P(U,X) =
\varinjlim\limits_{Z\in Fin(U,X)} \mathcal P(Z\to U)=
\varinjlim\limits_{Z\in Fin(U,X)} \prod\limits_{z\in Z}\mathcal P(Z^h_z\to U)=\\
\varinjlim\limits_{Z^\prime\in Fin(X)}
\varinjlim\limits_{Z\in Fin^{Z^\prime}(U,X)} 
\prod\limits_{z\in Z}\mathcal P(Z^h_z\to U)=\\
\varinjlim\limits_{Z^\prime\in Fin(X)}
\prod\limits_{x\in Z^\prime}
\varinjlim\limits_{Z\in Fin^x(U,X)} 
\mathcal P(Z\to U)=
\varinjlim_{Z^\prime\in Fin(X) }
\prod\limits_{x\in Z^\prime}\mathcal P(U,X^h_x)
,\end{multline*}
and consequently 
\begin{multline*}
GWCor(U,X)=
GW(\mathcal P(U,X))=
\varinjlim_{Z^\prime\in Fin(X)}
\prod\limits_{x\in Z^\prime}GW(\mathcal P(U,X^h_x))=\\
\bigoplus\limits_{x\in X}GW(\mathcal P(U,X^h_x))=
\bigoplus\limits_{x\in X}GWCor(U,X^h_x)
.\end{multline*}
and similarly for $WCor_k$.
\end{proof}
Further in the text we use the following notation
\begin{definition}
Let's denote the representable presheaves $\bZ_{GW}(X)=GWCor(-,X)$, and $\bZ_{W}(X)=WCor(-,X)$,
and denote by $\bZ_{GW,nis}(-)$, and $\bZ_{W,nis}(-)$
the associated Nisnevich sheaves.
\end{definition}

\section{Transfers for Nisnevich cohomolosies. }\label{sect:GWNiscoh}

In this section we prove the commutativity of the first square in diagram \eqref{diag:Constr}, 
and deduce that 
the category $\ShNGW$ 
of sheaves with GW-transfers\footnote{similarly for Witt-transfers} 
is abelian, and there is a natural isomorphism
$Ext^i_{\ShNGW}(\bZ_{GW,nis}(X),F[i])=H^i_{nis}(X,F)$
for $X\in Sm_k$, $F\in \ShNGW$, and similarly for Witt-transfers.
The reason of such behaviour of GW-correspondences (Witt-correspondences) is 
Actually the same is true for any category of correspondences (ringoid) $\Cor$ satisfying
relation \eqref{eq:CorvNisLoc}. 

Given a ringoid $\Cor$,
let $Sh_{tr}=\ShNTrs=\ShNTr$ denote the category of sheaves with $\Cor$-transfers, i.e. additive functors $\Cor\to Ab$ such that the restriction to a functor $Sm_k\to Ab$ is a sheaf. In the rest of the section if we write the words 'correspondences' and 'transfers' with out a prefix we mean the $\Cor$-correspondences and $\Cor$-transfers. 

\begin{theorem}
Suppose $\Cor$ is a ringoid
such that equation \eqref{eq:CorvNisLoc} holds.
Then $\ShNTrs$ is abelian
and 
for any presheaf with $\Cor$-transfers the associated Nisnevich sheaf and the cohomology presheaves are equipped with $\Cor$-transfers in a canonical way,
and moreover,
there is a natural isomorphism
$$Ext^i_{\ShNTr}(\bZ_{tr,nis}(X),F) = H^i_{nis}(X,F) ,$$
for any sheaf with transfers $F$, where $\bZ_{tr,nis}(X)$ denotes the Nisnevich sheafification of the representable presheaf $\bZ_{tr}(X)=\Cor(-,X)$.
\end{theorem}

\begin{definition}\label{def:cat-denot}
Denote by $Pre$ the category of presheaves of abelian groups on $Sm_k$,
and denote by $\ShN=\ShNis$ the category of Nisnevich sheaves on $Sm_k$.
Similarly, we denote by $\PreTrs=\PreTr$ the category of presheaves with $\Cor$-transfers, 
i.e. additive functors $Cor_{virt}\to Ab$, 
and as noted above we denote by $\ShNTrs=\ShNTr$ the full subcategory of $\PreTrs$ 
spanned by such presheaves that are Nisnevich sheaves.
\end{definition}
\begin{definition}\label{def:Dnisloc}
The morphism $w\colon A^\bullet \to B^\bullet$
  in the derived category $\Dba(Pre)$ or $\DbaPreTr$
is called \emph{Nisnevich quasi-isomorphism} or \emph{sheaf quasi-isomorphism} iff 
the homomorphisms
  $$\underline h^i(w)\colon 
    \underline h^i(A^\bullet) \to \underline h^i(B^\bullet), i\in \mathbb Z,$$
  of the Nisnevich sheaves associated with cohomology presheaves ($h^i(A^\bullet)$ and $h^i(B^\bullet)$) are isomorphisms.
A complex $A^\bullet$ 
  in the category $\Dba(Pre)$ or $\DbaPreTr$ 
is called \emph{sheaf acyclic} iff
all Nisnevich sheaves $\underline h^i(A^\bullet)$ 
are zero.

Denote as $\DbaNisTr$ the localisation of the category $\DbaPreTr$ in respect to Nisnevich quasi-isomorphisms.
\end{definition}
\begin{remark}
Let's note also that
we consider here bounded above derived categories, though the reasoning works as well for unbounded derived categories.
\end{remark} 

Consider the adjunction $\LtrPre\colon Pre \dashv \PreTrs \colon \FtrPre$, where 
$\LtrPre$ is the left Kan extension and $\FtrPre$ is forgetful.
Then by definition $\FtrPre$ is exact, and the induced functor on the derived categories preserves Nisnevich quasi-isomorphisms. 
We {\em claim} that the left derived functor $\LtrDPre=\mathcal L L_{tr}\colon \Dba(Pre)\to \DbaPreTr$ preserves Nisnevich quasi-isomorphisms. 

To prove the claim we use the description of the Nisnevich acyclic complexes via the category $\HLoc$ 
of essential smooth local henselian schemes
and the category $\CorHLocs=\Cor(\HLoc)$ of correspondences between such schemes. 
%
Consider the diagram
\begin{equation}\label{eq:LocDiag}
\xymatrix{
Pre\ar[d]^{\loc}\ar[r]^{\LtrPre} & \PreTr\ar[d]^(0.4){\loctr}\\
\PreHLoc\ar[r]^{\LtrLoc} & \PreHLocTrs 
}
\xymatrix{
\DbaPre\ar[d]_{\locD}\ar[r]^{\LtrDPre} & \DbaPreTr\ar[d]^{\locDtr}\\
\DbaPreHLoc\ar[r]^{\LtrLocD} & \DbaPreHLocTrs
,}\end{equation}
where 
$\PreHLocs=\PreHLoc$  and $\PreHLocTrs=\PreHLocTr$ are 
categories of presheaves on $\HLoc$ and $\CorHLoc$,
$\loc$ and $loc_{tr}$ are the restriction functors, which are exact,
and $\LtrLocD$ and $\LtrDPre$ are left derived of $\LtrLoc$ and $\LtrDPre$.

The subcategories of Nisnevich acyclic objects in $\DbaPre$ and $\DbaPreTr$ are exactly the kernels of the functors $\loc$ and $\loctr$. 
So the claim follows from the commutativity of the diagrams above. In fact, the commutativity of the diagrams above 
is exactly what equation \eqref{eq:CorvNisLoc} states.
We realise this strategy with more details in the following sequence of lemmas.

\begin{lemma}\label{lm:WCorsplitting}
Suppose $Cor_{virt}$ is a ringoid such that equation \eqref{eq:CorvNisLoc} holds;
then for any $X\in Sm_k$,
$$
loc(\bZ(X))=\bigoplus\limits_{x\in X}\bZ(X^h_x),\quad
loc_{tr}(\bZ_{tr}(X))=\bigoplus\limits_{x\in X}\bZ_{tr}(X^h_x)
,$$
where $\bZ_{tr}(X)=Cor_{virt}(-,X)$.
\end{lemma}\begin{proof}
The first equality follows from that for any morphism $U\to X$, where $U$ is local henselian, there is a lift to a morphism $U\to X^h_x$.
The second one is just a reformulation of equality \eqref{eq:CorvNisLoc}.
\end{proof}
\begin{lemma}\label{lm:WCorNiscommutsq}
Suppose $Cor_{virt}$ is a ringoid such that equation \eqref{eq:CorvNisLoc} holds;
then the squares \eqref{eq:LocDiag} are commutative.
\end{lemma}\begin{proof}

Since the compositions $Loc\to Sm_S\to \Cor$ and $Loc\to \Cor^{Loc}\to \Cor$ coincides,
it follows that $loc\circ F_{tr}= F_{tr}^{loc}\circ loc_{tr}$,
where $F_{tr}$ and $F_{tr}^{loc}$ are forgetful functors.
The unit and counit of the adjunctions $L_{tr} \dashv F_{tr}$, $L^{loc}_{tr} \dashv F^{loc}_{tr}$
gives sequence on natural transformations
$$loc_{tr}\circ L_{tr}\leftarrow L_{tr}^{loc}\circ F_{tr}^{loc}\circ loc_{tr}\circ L_{tr} \simeq L_{tr}^{loc}\circ loc\circ F_{tr}\circ L_{tr}\leftarrow L_{tr}^{loc}\circ loc.$$
Denote the composition by $\nu\colon loc\circ L_{tr}^{loc}\to loc_{tr}\circ L_{tr}$.
For a representable presheaves $\bZ(X)$ we have $$(L_{tr}^{loc}\circ loc)(\bZ(X))=\bigoplus\limits_{x\in X}( \bZtr(X^h_x)) ), loc_{tr}\circ L_{tr}(\bZ(X))=loc(\bZtr(X))$$ and the morphism
$$\nu(\mathbb Z(X))\colon \bigoplus\limits_{x\in X}( \bZtr(X^h_x)) )\to loc(\bZtr(X))$$ is equal to the morphism which is defined by compositions with morphisms $X^h_x\to X$.
By lemma \ref{lm:WCorsplitting} the last homomorphism is isomorphism, and so we get that $\nu$ is isomorphism on representable presheaves.
Now since any presheaf in $Pre$ is a colimit of representable presheaves, the claim follows.

To prove commutativity of the right square let's note that 
(1) representable presheaves are projective objects in categories of presheaves,
(2) for any complex of presheaves there is representable resolvent, 
(3) functors $loc$ 
takes representable presheaves to representable ones;
hence
by the theorem about the derived functor of a composition of functors we have 
$\LtrLocD\circ loc = \mathcal L(\LtrLocD\circ loc)$, $loc_{tr}\circ \LtrPre = \mathcal L(loc_{tr}\circ \LtrPre)$.
\end{proof}

\begin{theorem}\label{th:WtrNisEx}
Suppose $Cor_{virt}$ is a ringoid such that equation \eqref{eq:CorvNisLoc} holds;
then the functor
$\LtrDPre\colon \Dba(Pre)\to \DbaPreTr$
is exact in respect to Nisnevich quasi-isomorphisms (or equivalently, $\LtrDPre$ takes
Nisnevich acyclic complexes in $\Dba(Pre)$ to Nisnevich acyclic ones in $\DbaPreTr$). 
\end{theorem}
\begin{proof}

Subcategories of Nisnevich acyclic complexes in $ \Dba(Pre)$ and $\DbaPreTr$ are exactly preimages 
under the functor $\loc$ and $\loctr$ of the categories of 
acyclic complexes in $\PreHLoc$ and $\PreHLocTrs$.
So the claim follows form
the commutativity of the right diagram in \eqref{eq:LocDiag} proved by lemma \ref{lm:WCorNiscommutsq}.
\end{proof}

\begin{corollary}
Suppose $Cor_{virt}$ is a ringoid such that equation \eqref{eq:CorvNisLoc} holds;
then there is an adjunction
$$\LtrDShN \colon \Dba(\ShN) \leftrightharpoons \DbaNisTr\colon \FtrN$$
such that 
$\FtrN$ takes a complex of presheaves with $Cor_{virt}$-transfers to itself considered as a complex of presheaves of abelian groups,
and $\LtrDShN(\bZ(X))=\bZ_{tr}(X)$. 
\end{corollary}
\begin{proof}
Obviously we have adjunction $\LtrDPre\colon \Dba(Pre) \leftrightharpoons \DbaPreTr\colon \FtrPre$,
and by definition $\LtrDPre(\bZ(X))=\bZ_{tr}(X)$ and 
$\FtrPre$ takes a complex of presheaves with transfers to 
itself considered as a complex of presheaves of abelian groups.
%
Now 
since by def. \ref{def:Dnisloc} and theorem \ref{th:WtrNisEx}
the functors $\FtrDPre$ and $\LtrDPre$ preserve Nisnevich acyclicity,
the claim follows form (see \cite[Appendix, lemma 6.5]{AD_DMGWeff}).
%
\end{proof}
\begin{corollary}\label{lm:HomNisloc=CohNis}
Suppose $Cor_{virt}$ is a ringoid such that equation \eqref{eq:CorvNisLoc} holds;
then for any $X\in Sm_k$ and a presheaf with transfers $F$ there is a natural isomorphism
\begin{equation}\label{eq:TrNislocadjNisloc} Hom_{\DbaNisTr}(\bZ_{tr}(X),F[i])=H^i_{nis}(X,F) .\end{equation}
\end{corollary}
\begin{proof}\end{proof}

Thus we get the diagram
\begin{equation}\label{diag:nisloc}\xymatrix{
\Dba(Pre) \ar@<1ex>[r]^(0.4){\LtrDPre}\ar@<1ex>[d]^{\lnis} & 
\Dba(\PreTr)\ar@<1ex>[d]^{\lNisTr}\ar@<1ex>[l]^(0.6){\FtrDPre}
\\
\Dba(\ShN)\ar@<1ex>[r]^(0.4){\LtrDShN} & 
\DbaNisTr\ar@<1ex>[l]^(0.55){\FtrN}
}\end{equation}
with both squares being commutative.
Now using transfers defied by the functor $\Cor\to \DbaNisTr$ and equality \eqref{eq:TrNislocadjNisloc} we can show that $\ShNTrs$ is abelian and $\DbaNisTr\simeq\DbaShNTrs$.
\begin{corollary}\label{cor:NisCohTr}
Suppose $Cor_{virt}$ is a ringoid such that equation \eqref{eq:CorvNisLoc} holds;
%
then 
the category $\ShNTrs$ is abelian
and there is a reflection $$\LnisTr\colon \PreTr\leftrightharpoons \ShNTr\colon \RnisTr$$ such that 
the right adjoint functor $\RnisTr$ is the embedding functor, the left adjoint $\LnisTr$ is exact and for any $F\in \PreTr$ the sheaf $\LnisTr(F)$ is canonically isomorphic to $F_{nis}$ as a sheaf of abelian groups.
\end{corollary}
\begin{proof}
Corollary \ref{lm:HomNisloc=CohNis} and the functor $Cor_{virt}\to \PreTr\to \DbaPreTr\to \DbaShNTr$ implies that for any $F\in \PreTr$ the sheaf $F_{nis}$ is equipped with transfers in a natural way.
This defines the functor $\LnisTr$. 

To prove that there is a reflexive adjunction $\LnisTr\dashv\RnisTr$ it is enough to show that the canonical morphism $\nu_F\colon F\to F_{nis}$ is compatible with transfers, since this defies co-unit of the required adjunction, while the unit is the identity $\id_{\ShNTr}$.
%
Diagram \eqref{eq:LocDiag} leads to the commutative square of Hom-groups
$$\xymatrix{
F(X)\ar[d]^{\nu_F}
\ar[d]\ar[r]^(.25){\simeq} 
& Hom_{\DbaPreTr}(\bZ_{tr}(X),F)\ar[d] \\
F_{nis}(X)
\ar[r]^(.25){\simeq} & Hom_{\DbaShNTr}(\bZ_{tr}(X),F)  ,
}$$
such that 
the left vertical arrow is the morphism $\nu_F$ and the right one is the morphism of presheaves with transfers.
%

Now straightforward verification shows that formulas 
$$ \Ker_{\ShNTrs}(f) = \Ker_{\PreTrs}(f),\quad \Coker_{\ShNTrs}(f)=\LnisTr(\Coker_{\PreTr}(f)),$$
defines kernel and cokernel
for any $f\colon F_1\to F_2\in \ShNTr$.
Then 
the forgetful functor $\FtrN\colon \ShNTrs\to \ShN$ is exact. 
So since it is $\FtrN$ conservative, it follows that $\ShNTrs$ is abelian.

Finally, let's note that functors $L_{nis}^{Pre}$ and $F_{tr}^{\PreTrs}$ are exact. Whence, since 
$F_{tr}^{\ShNTrs} \circ \LnisTr = L_{nis}^{Pre}\circ F_{tr}^{\PreTrs} \colon \PreTr\to \ShN$ is exact, 
it follows that $\LnisTr$ is exact.

%
\end{proof}
\begin{lemma}\label{lm:DnisDSh}
There is an equivalence $\DbaNisTr\simeq \DbaShNTr$,
and the functor $\DbaPreTr \to \DbaShNTr$ induced by the exact functor $\LnisDTr$ is equivalent to the functor $\lnistr$ from diagram \eqref{diag:nisloc}.
\end{lemma}
\begin{proof}
By corollary \ref{cor:NisCohTr} the functor $\LnisTr$ is exact. Hence it induces the functor on the derived categories.
It follows form definition \ref{def:Dnisloc} that $\Ker(\LnisDTr)$ is the subcategory of Nisnevich acyclic complexes in $\DbaPreTrs$. Whence $\LnisDTr$ can be passed throw a functor $\DbaNisTrs\to \DbaShNTrs$.

Conversely, the functor $\RnisTr(f)$ form \ref{cor:NisCohTr} induces the functor $\RnisTr(f)\colon K^-(\ShNTrs)\to K^-(\PreTrs).$ It  follows form definition \ref{def:Dnisloc} that $\RnisTr$ takes quasi-isomorphisms in $\DbaShNTrs$ to Nisnevich-quasi-isomorphisms in $\DbaNisTrs$. Hence this defines the inverse functor $\DbaShNTrs\to \DbaNisTrs$\footnote{Note that $\RnisTr$ doesn't induce a functor $\DbaShNTrs\to\DbaPreTrs$.}.
\end{proof}

\begin{theorem}\label{th:ExtTr=Hnis}
Suppose $Cor_{virt}$ is a ringoid such that equation \eqref{eq:CorvNisLoc} holds;
then the category of sheaves with transfers $\ShNTr$ is abelian and 
\begin{equation}\label{eq:ExtNCoh}Ext^i_{\ShNTr}(\bZ_{tr,nis}(X),F)=H^i_{nis}(X,F).\end{equation}
\end{theorem}
\begin{proof}
It is already proven in corollary \ref{cor:NisCohTr}
that $\ShNTr$ is abelian,
and 
since $\LnisTr(\bZ_{tr}(X))=\bZ_{tr,nis}(X), $
isomorphism \eqref{eq:ExtNCoh} follows form isomorphism \eqref{eq:TrNislocadjNisloc}
via lemma \ref{lm:DnisDSh}. 
\end{proof}

By corollary \ref{cor:NisCohTr} the sheaves $\bZ_{tr,nis}(X)$ are equipped with GW-transfers in a natural way. This gives a sense to formula of effective motives \eqref{eq:intr:MGW}.  
Let's show that the defined structure of transfers on $\bZ_{tr,nis}(X)$ is the unique structure 
such that the canonical homomorphism $\bZ_{tr}(X)\to \bZ_{tr,nis}(X)$ is a homomorphism of presheaf with transfers.

\begin{proposition}\label{prop:uniqnistrans}
Suppose $\Cor$ is a ringoid such that equation \eqref{eq:CorvNisLoc} holds;
then for any presheaf with $\Cor$-transfers $F$ 
there is a unique 
structure of a presheaf with transfers on 
$F_{nis}$ 
such that 
the canonical homomorphism $\nu\colon F\to F_{nis}$
is a morphism of presheaves with transfers. 
%

\end{proposition}
\begin{proof}
The existence is already proved in corollary \ref{cor:NisCohTr}.
Let's show the uniqueness. 
Let $G_1$ and $G_2$ be a pair of presheaves with 
morphisms $F \xrightarrow{\eta_i} G_i \stackrel{\tau_i}{\simeq} F_{nis}$, $i=1,2,$ 
where $\eta_i$ are morphisms in $\PreTr$ and $\tau_i$ are isomorphisms in $Pre$ such that $\tau_i\circ\eta_i=\nu$.
The claim is to prove 
that $\tau_1( G_1(\Phi)(\tau_1^{-1}(a)) )=\tau_2( G_2(\Phi)(\tau_2^{-1}(a)) )$
for any $X,Y\in Sm_k$, $\Phi\in Cor_{virt}(X,Y)$, and $a\in F_{nis}(Y)$.

Let $v\colon \mathcal V\to Y$ be a Nisnevich covering and let $\widetilde a\in F(\mathcal V)\colon$ $\nu(\widetilde a)=v^*(a)$. 
It follows form equality \eqref{eq:CorvNisLoc} that there is a Nisnevich covering $u\colon \mathcal U\to X$ and $\widetilde\Phi\in Cor_{virt}(\mathcal U,\mathcal V)$
such that $v\circ\widetilde\Phi=\Phi\circ u\in Cor_{virt}(\mathcal U,Y))$.
Then 
$$u^*(G_i(\Phi)(\tau_i(a))) = G_i(\Phi)(\tau_i(v^*(a))) = G_i(\widetilde \Phi)(\eta_i(\widetilde a)) = \eta_i( F(\widetilde \Phi)(\widetilde a) ) 
,
$$
and whence $$u^*( \tau^{-1}_1( G_1(\Phi)(\tau_1(a)) ) ) = \tau^{-1}_1( u^*(G_1(\Phi)(\tau_1(a))) ) = \tau^{-1}_1( \eta_i( F(\widetilde \Phi)(\widetilde a) ) ) = \nu( F(\widetilde \Phi)(\widetilde a) ) .$$
Now since $v$ is a Nisnevich covering and homomorphism $v^*\colon F_{nis}(Y)\to F_{nis}(\mathcal V)$ is injective,
the claim follows.
%
%
%
%
%
\end{proof}

We finish the section by the following
\begin{corollary}\label{cor:GWNisKerGen}
Suppose $\Cor$ a ringoid such that equation \eqref{eq:CorvNisLoc} holds;
then 
the subcategory of Nisnevich-acyclic complexes in $\DbaPreTrs$ is  the thick subcategory generated by complexes 
\begin{equation}\label{eq:cNissq}
Tot(\mathbb Z_{tr}(\mathcal N))=[\dots\to 0\to \mathbb{Z}_{tr}(\tilde U)\rightarrow \mathbb{Z}_{tr}(U)\oplus \mathbb{Z}_{tr}(\tilde X)\rightarrow \mathbb Z_{tr}(X)\to \dots] 
\end{equation}
for all Nisnevich squares $\mathcal N$:
\begin{equation}\label{eq:Nissq}
\xymatrix{
\tilde U\ar[r]\ar[d]& \tilde X\ar[d]\\
U\ar[r] & X
}
\end{equation}  
\end{corollary}
\begin{proof}
We deduce the claim form the similar statement for the category 
of presheaves without transfers. 

Let $\Tau$ and $\Tau^{tr}$ be localising subcategories in $\DbaPre$ and $\DPreTrs$ generated by complexes \eqref{eq:cNissq}.
The claim is equivalent to that $Hom _{ \DbaPreTrs/\Tau^{tr}}(\bZ_{tr}(X),F^\bullet) = 0$
for any Nisnevich acyclic complex $F^\bullet\in \DbaPreTrs$ and any smooth scheme $X$.

It follows from the definition that $\mathcal L_{tr}(\Tau)=L_{tr}(\Tau)\subset \Tau^{tr}$, and it follows form theorem \ref{th:WtrNisEx}
that $F_{tr}(\Tau^{GW})\subset \Tau$. So the adjunction $\mathcal L_{tr} \dashv F_{tr}$ induces the adjunction 
$$\mathcal L_{tr}\colon \DbaPre/\Tau\leftrightharpoons \DbaPreTrs/\Tau^{tr}\colon F_{tr},$$
and for any $X\in Sm_k$ and $F^\bullet \in \DbaPreTrs$ we have
\begin{equation}\label{eq:adjGWtrTau}Hom _{ \DbaPreTrs/\Tau^{tr}}(\bZ_{tr}(X),F^\bullet) = Hom _{ \DbaPre/\Tau}(\bZ(X), F_{tr}(F^\bullet)).\end{equation}
Hence if $F^\bullet$ is Nisnevich-acyclic, then $F_{tr}(F^\bullet)$ is Nisnevich-acyclic by definition, 
and it follows that the right side of \eqref{eq:adjGWtrTau} is zero, 
since the family of complexes $Tot(\bZ(\mathcal N))$ 
generates $\Tau$ (\cite[Appendix, proposition 6.6]{AD_DMGWeff}).
(The complex $Tot(\bZ(\mathcal N))$ is defined like complex \eqref{eq:cNissq} with $\mathbb Z(-)$ instead of $\mathbb Z_{tr}(-)$.)
\end{proof}

\section{Categories of GW- and Witt-motives.} 
\label{sect:TensStr}

Since as noted in the introduction any homotopy invariant Nisnevich excisive cohomology theory with the transfers defined by 
a ringoid
$\Cor$ should be passed throw the 
category of effective motives,
it is natural to define the category of $\Cor$-motives as follows: 

\begin{definition}\label{def:DMeff}
Given a base scheme $S$, and a ringoid $\Cor$
over $S$,
define a category $\DMeffmCorvl$ with a functor $l_{\affl,nis}\colon \DbaPreTrs\to \DMeffmCorvl$ as
the localisation 
in respect to morphisms of the form $Tot(\mathbb Z_{tr}(\mathcal N)) \to 0$ for all Nisnevich squares $\mathcal N$ (see \eqref{eq:cNissq}, \eqref{eq:Nissq}) and morphisms of the form $X\times\affl\to X$, $X\in Sm_S$.
The functor $M^{\Cor}\colon Sm_k\to \DMeffmCorv$ given by $
l_{\affl,nis}(\bZ_{tr}(-))$ is called as \emph{effective motive of smooth schemes}. 
\end{definition}

%


%
In this section we define the tensor structure on $\DMeffmCorvl$. 
Formally the construction presented here 
differs from the construction from \cite{SV_Bloch-Kato} for the original case of $\mathbf{DM}(k)$, but these constructions lead to the same resulting tensor structure.

The tensor product functor 
is defined according to the following diagram
$$\xymatrix{
Sm_k\times Sm_k \ar[d]|{-\times -}\ar[r] & (\PreTrs)^{\times 2}\ar[d]|{-\otimes^{Pre} -}\ar[r] & 
\DbaPreTrs^{\times 2} \ar[d]|{-\otimes^{\mathbf{D}} -}\ar[r] & \DMeffmCorvl(k)^{\times2}\ar[d]|{-\otimes^{\mathbf{DM}} -}\\
Sm_k\ar[r]^{\mathbb Z_{GW}} & \PreTrs \ar[r] & \DbaPreTrs_k \ar[r]^{l_{nis,\affl}} & \DMeffmCorvl(k) 
}$$
where each next vertical arrow satisfies the left universal property (a left Kan extension or a left derived functor). More precisely, let's give the following definition:
\begin{definition}
1)
The functor $\otimes^{Pre}$ is the left Kan extension of the functor $(X,Y)\mapsto \mathbb Z_{tr}(X\times Y)$ with respect to the functor $Sm^2_k\to (\PreTrs)^{\times 2}$.
Precisely, the functor $\otimes^{Pre}$ can be defined by the formula
\begin{equation*}\label{eq:pretens}
\mathcal F_1\otimes\mathcal F_2 = 
\varinjlim\limits_{S_{\mathcal F_1,\mathcal F_2}} 
\mathbb Z_{tr}(U_1\times U_2),\;
\mathcal S_{\mathcal F_1,\mathcal F_2} = 
\{(a_1,a_2,U_1,U_2)\colon a_1\in \mathcal F(U_1), a_2\in \mathcal F(U_2)\}
,\end{equation*}
$S_{\mathcal F_1,\mathcal F_2}((b_1,b_2,V_1,V_2),(a_1,a_2,U_1,U_2)) = \{(f_1,f_2)|f_i\in \Cor(V_i,U_i), b_i = f_i^*(a_i), i=1,2\}$.

2)
Consider the projective model structure on $K^-(\PreTrs)$ and corresponding model structure on $K^-(\PreTrs)^2$. 
Define $\otimes^{\mathbf D}$ as the left derived functor to the functor $\otimes^{Pre}$ 
that is section-wise application of $\otimes^{Pre}$.
Precisely, the functor $\otimes^{\mathbf D}$ is given by 
the formula 
$(A^\bullet,B^\bullet)\mapsto Tot(\tilde A^\bullet, \tilde B^\bullet)$,
where $\widetilde{ A}^\bullet\to A^\bullet$ and $\widetilde{ B}^\bullet\to B^\bullet$ are quasi-equivalences such that $\widetilde{ A}^\bullet$ and $\widetilde{ B}^\bullet$ are complexes with terms being direct sums of representable presheaves, in particular we can use representable resolvents produced by iterating of the procedure from the previous point.
\end{definition}
The functor $\otimes^{\mathbf{DM}}$ can be defined in a same way (as derived functor), but, indeed, the functor $\otimes^{\mathbf{D}}$ preserves Nisnevich- and $\affl$-quasi-iso\-mor\-phisms and so it just defines a functor on the localisations. 

\begin{proposition}\label{prop:MotiveProd}
The functor $\otimes^{\mathbf D}$ is exact in respect to the localisation $$l_{\affl,nis}\colon \DbaPreTr\to \DMeffmCorvl,$$ 
and 
$$M^{\Cor}(X)\otimes^{\mathbf{DM}}M^{\Cor}(Y)=M^{\Cor}(X\times Y),$$
where $\otimes^{\mathbf{DM}}\colon ({\DMeffmCorvl(k)})^2\to \DMeffmCorvl(k)$  denotes the induced functor.
\end{proposition}
\begin{proof}
Since representable presheaves $\mathbb Z_{tr}(X)$ for $X\in Sm_k$ are projective objects in $\PreTrs$,
it follows that $\mathbb Z_{tr}(X)\otimes^{\mathbf D} \mathbb Z_{tr}(Y) = \mathbb Z_{tr}(X\times Y)$.
So 
since the functors $-\times X$ and $X\times -$ preserves morphisms of the form $\affl\times U\to U$ and Nisnevich squares,
the claim follows.
\end{proof}
Using the tensor structure on the category $\DMeffmCorvl$ 
we can define the category of (non-effective) $\Cor$-motives $\mathbf{DM}^{\Cor}$.

\begin{definition}
We define $\bGm^{\wedge 1} = \Cone(pt\stackrel{1}{\hookrightarrow} \bGm)$
,
$\bGm^{\wedge k} = (\bGm^{\wedge 1})^{\otimes k}$
,
and for any $A^\bullet\in D^{W}_{eff}(k)$, we denote $A^\bullet(n) = \bGm^{\wedge k}\otimes A^\bullet$.
\end{definition}

\begin{definition}
Define the category of (non-effective) motives $\mathbf{DM}^{\Cor,-}(k)$ 
by inverting of $-\otimes\bGmw$
in  $\DMeffmCorvl(k)$. 
So objects of $\DMeffmCorvl(k)$ are 
$M(n)$ where $M\in \DMeffmCorvl(k)$ and $n\in \mathbb Z$,
and $Hom(M_1(n_1),M_2(n_2)) = \varinjlim\limits_{i\in \mathbb Z}(M_1\otimes\bGm^{\wedge i+n_1}, M_1\otimes\bGm^{\wedge i+n_2})$.
A functor 
$\Sigma^\infty_{\bGm}\colon \DMeffmCorvl(k) \to \mathbf{DM}^{\Cor,-}(k)$ takes a motivic complex
$M$  to the spectrum  $M(0)$.

%
\end{definition}

\begin{definition}
Define the category of (non-effective) motives $\mathbf{DM}^{\Cor,-}(k)$ 
as the homotopy category of $\bGmw$ spectra over the category $\DMeffmCorvl(k)$

So firstly consider the category $Sp_{\bGm}(\DMeffmCorvl(k))$ with 
objects being 
sequences
$$E = (E_0^\bullet,s_0,E_1^\bullet,s_1,\dots,E_n^\bullet,s_n,\dots), E_i^\bullet\in \DMeffmCorvl(k),
s_i\in Hom_{\DMeffmCorvl(k)}(E_i^\bullet(1),E_{i+1}^\bullet)$$ 
and 
$Hom_{Sp_{\bGm}(\DMeffmCorvl(k))}(E,F) = \varinjlim_i Hom_{\DMeffmCorvl(k)}(E_i,F_i)$
for any $E= (E_0,\dots, E_i,\dots)$ and $F = (F_0,\dots, F_i,\dots)$.
Now define the category $\mathbf{DM}^{\Cor,-}(k)$ as the localisation of the category $Sp_{\bGm}(\DMeffmCorvl(k))$ in respect to the localizing class of morphisms $f\colon E\to F\in \DMeffmCorvl(k)$ such that for any $X\in Sm_k$ the morphism $f$ induces isomorphisms $$Hom(M^{GW}_{eff}(X),E)\simeq Hom(M^{GW}_{eff}(X),F).$$

A functor 
$\Sigma^\infty_{\bGm}\colon DM^{GW}_{eff}(k) \to DM^{GW}(k)$ takes a motivic complex
$E^\bullet$ to the spectrum with terms $E^\bullet(i)$ and homomorphisms $id_{E^\bullet(n)}$.
\end{definition}

\section{Effective motives of smooth varieties}\label{sect:DMeff}

It follows form 
lemma \ref{lm:DnisDSh} and definition \ref{def:DMeff}
that if a ringoid $\Cor$ satisfies condition \eqref{eq:CorvNisLoc},
then the category of effective motives $\DMeffmCorvl$ is equal to the localisation of $\DbaShNTrs$ in respect to morphisms of the form $X\times\affl\to X$.

Moreover, 
if we assume in addition 
that $\Cor$ satisfies the \emph{strictly homotopy invariance axiom}: 
\begin{itemize}[leftmargin=5pt]
\item[]\emph{(SHI) for any homotopy invariant presheaf $F\in \PreTrs$ 
the presheaves $h_{nis}^i(F_{nis})$ are hom. inv. for all $i\geq 0$,}
\end{itemize}
then the category $\DMeffmCorvl$
is equivalent to 
the full subcategory $\DMeffmCorvr$ of $\DbaShNTrs$
spanned by motivic complexes. 
The category $\DMeffmCorvr$ can be considered as the computation for $\DMeffmCorvl$, 
since hom-groups in $\DMeffmCorvr$ are equal by definition to the hom-groups in the derived category $\DbaShNTrs$, which are simpler for computations.

Let's note also that the equivalence $\DMeffmCorvl\simeq \DMeffmCorvl$ follows from the original methods of the work \cite{VSF_CTMht_DM} 
or form the general technique presented in \cite{GarPan-Kmot14}.
But 
we present here 
another proof of this result via the semi-orthogonal decompositions on $\mathbf{D}(\PreTrs)$,
since the additional and intermediate results in our proof are also useful for the aim of the cancellation theorem for GW- and Witt-correspondences, which as noted above is not the content the of this article nevertheless.

Informally, the reason why 
the strictly homotopy invariance axiom 
yields that $\DMeffmCorvr\simeq\DMeffmCorvl$
is that it implies some coherence of two structures 
combined in the category of effective motives, namely,
Nisnevich topology and the structure of category with interval on $Sm_k$  induced by $\affl$;
in terms of semi-orthogonal decompositions the "coherence" above means, in fact, that
the Nisnevich topology and the 'action' of affine line induce compatible semi-orthogonal decompositions on $\mathbf{D}(\PreTrs)$.
We refer reader to \cite{Neem_TriangCat} and \cite{Krause_localizationinTriangCat} for the used facts about localisations and semi-orthogonal decompositions in 
triangulated categories. 
\vspace{3pt}



Consider the endo-functor $F\mapsto F(-\times \affl)$ on the abelian category $\PreTrs$. Since it is exact, it induce the endo-functor 
on the derived category $\DbaPreTrs$. Let's denote this end-functor as $\mathcal Hom(\affl,-)$; denote by $s_0\colon S\to Id$, $s_1\colon S\to Id$ the natural transformations induced by the inclusions $pt\hookrightarrow \affl$ by zero and unit points respectively;
and denote by $p\colon Id\to S$ the natural transformation induced by the canonical projection $\affl\to pt$.

\begin{definition}\label{def_afficObj}
An object $A^\bullet\in \DbaPreTr$ is called $\affl$-\emph{invariant} 
whenever the morphism $p(A^\bullet)\colon A^\bullet\to S(A^\bullet)$ is isomorphism,
and 
$A^\bullet$ is called $\affl$-\emph{contractable} 
whenever there is a morphism $h\colon A^\bullet \to S(A^\bullet)$ such that 
$s_0\circ h=\id_{A^\bullet}$ and $s_1\circ h=0$.

Let 
$\iaffDbaP\subset \DbaPreTr$
be the full triangulated subcategory spanned by homotopy invariant objects,
and 
$\caffDbaP\subset \DbaPreTr$ 
be the thick subcategory generated by $\affl$-contractable objects.  
\end{definition}

%
%

\begin{lemma}\label{lm:contrinvort}
Suppose $A^\bullet\in \iaffDbaP$ and $B^\bullet\in \caffDbaP$,
then $$Hom_{\DbaPreTrs}(B^\bullet,A^\bullet)=0.$$
\end{lemma}
\begin{proof}
It is enough to consider the case of $B^\bullet$ being $\affl$-contractable.
Since $A^\bullet$ is $\affl$-invariant object, 
 we have $s_0(A^\bullet)= s_1(A^\bullet)\colon S(A^\bullet)\to A^\bullet.$
On other side since
  $B^\bullet$ is $\affl$-contractable object,
it follows that there is a morphism
  $$
  h \in Hom_{\DbaPreTrs}(B^\bullet, \mathcal S(B^\bullet)\colon\quad
  s_0(B^\bullet)\circ h = id,\; s_1(B^\bullet)\circ h = 0
  .$$

Hence
  for any $f\in Hom_{\DbaPreTrs}(B^\bullet,A^\bullet)$
  $$
  f =
  f\circ s_1(B^\bullet)\circ h =
  s_0(A^\bullet)\circ S(f)\circ h =
  s_1(A^\bullet)\circ S(f)\circ h = 
  f\circ s_0(B^\bullet)\circ h
  = 0
  .$$
$$\xymatrix{
S(B^\bullet)
  \ar@<-1ex>[d]_{s_1}\ar@<1ex>[d]^{s_0} \ar[r]^{S(f)}&
S(A^\bullet)
  \ar@<-1ex>[d]_{s_1}\ar@<1ex>[d]^{s_0} \\
B^\bullet
  \ar[u]|h \ar[r]^f &  
A^\bullet  
}$$  
\end{proof}

Next let's give the following standard definition.
\begin{definition}\label{def:affsimpl}
Let's denote by $\Delta$
  the simplicial scheme 
  with 
\begin{gather*}
\Delta^n = Spec k[x_0,x_1,\dots,x_n]/(x_0+x_1+\dots x_n -1) \\
e_{n,i} \colon
  (x_0,x_1,\dots,x_n)\mapsto (x_0,\dots,x_i,0,x_{i+1},\dots, x_n)\\
d_{n,i}\colon
  (x_0,x_1,\dots,x_n)\mapsto (x_0,\dots,x_i+x_{i+1},\dots, x_n)
.
\end{gather*}

Denote by $\pt$ the complex concentrated in the term zero degree with group of chains being $\bZtr(pt)$,
and denote 
$$\pt^\bullet=[ \bZtr(pt) \xrightarrow{id}\bZtr(pt)\xrightarrow{0}\bZtr(pt)\xrightarrow{id}\dots\xrightarrow{0}\bZtr(pt)\xrightarrow{id}\bZtr(pt)\xrightarrow{0} \stackrel{\deg 0}{\bZtr(pt)}\to 0\to 0\dots ]. $$ 
The embedding $\pt\to \pt^\bullet$ is quasi-isomorphism. 

Denote the endofunctor
$$C^* = \mathcal Hom_{\DbaPreTrs}(\Delta^\bullet, -)\colon \DbaPreTr\to \iaffDbaP$$
that takes the complex with terms being presheaves $A^i(-)$ to the complex with terms $A^i(\Delta^i\times -)$.
\end{definition}
\begin{remark}
The functors $\mathcal Hom(\affl,-)$, $\mathcal Hom(\Delta^i,-)$, and $\mathcal Hom(\Delta^\bullet,-)$ above, indeed, are 
the internal Hom-functor in the categories $\PreTrs$ and  $\DbaPreTrs$;
\end{remark}

\begin{theorem}\label{affsodDPW}
Suppose $\Cor$ as any ringoid over a base $S$%
; 
then
\begin{itemize}[leftmargin=15pt]
\item[1)]
the subcategory $\iaffDbaP$ is the full subcategory 
spanned by complexes with homotopy invariant cohomologies, 
it is the thick subcategory generated by homotopy invariant presheaves;
\item[2)]
the subcategory $\caffDbaP$ is 
the full subcategory of  $\DbaPs$ 
spanned by complexes quasi-isomorphic to 
    complexes of $\affl$-contractable presheaves; 
and it is is the thick subcategory generated 
by presheaves $\bZ_{tr}(\affl\times X)/\bZ_{tr}(X)$ for all $X\in Sm_k$.;
\item[3)]
there is
the semi-orthogonal decomposition 
  $$\DbaPs = \langle\iaffDbaP,\caffDbaP\rangle;$$
the left adjoint functor to the embedding $\iaffDbaP\to \DbaPs$ 
is defined by the functor $C^*$ (def. \ref{def:affsimpl}). 
\end{itemize}
\end{theorem}
\begin{proof}
Denote $\mathcal C=\DbaPs$, $\mathcal A=\iaffDbaP$, $\mathcal B=\caffDbaP$.

1)
Since the functor $\mathcal Hom_{\PreTrs}(\affl,-)$ is exact,
it follows that 
$$h^i(\mathcal Hom_{\mathcal C}(\affl,C^\bullet))  = h^i(\mathcal Hom_{\PreTr}(\affl,C^\bullet)) = h^i(C^\bullet).$$
Whence
homotopy invariant objects of $\mathcal C$ are
complexes with homotopy invariant cohomologies.
The last statement in the point 1 follows form that any complex in $\DbaPs$ is contained in the category generated by its terms via direct sums and cones.

2a) Any complex $B^\bullet\in \mathcal C$ with terms contractable presheaves is an object of the subcategory
$\mathcal B$, 
since $\affl$-contractable presheaves 
are $\affl$-contractable objects in $\mathcal C$, 
and
any complex is 
generated by its terms via (infinite) direct sums and cones.

3a)
Lemma \ref{lm:contrinvort} implies that 
$Hom_{\mathcal C}(B^\bullet,A^\bullet)=0$ for all $B^\bullet\in\mathcal B$ and $A^\bullet\in\mathcal A$.
To get the semi-orthogonal decomposition
it is enough to show that 
for any $C^\bullet\in \mathcal C$ there is a distinguished triangle 
$B^\bullet \to C^\bullet \to A^\bullet$ with $B^\bullet \in \mathcal B$, $A^\bullet \in \mathcal A$.

Consider the endo-functor 
$C^* = \mathcal Hom_{\mathcal C}(\Delta^\bullet, -)$ on $\DbaPreTrs$.          
%
Let 
$\varepsilon^\prime \colon C^*\to \id_{\mathcal C}$ 
denote the composition 
of natural transformations 
$
C^*\xleftarrow{\varepsilon} 
\mathcal Hom_{\mathcal C}(\pt^\bullet, -)\simeq 
\id_{\mathcal C}
$
induced by morphisms of simplicial objects $\Delta^\bullet\to pt^\bullet\xrightarrow{\sim} pt$.
Then for any complex $C^\bullet$ there is the distinguished triangle
\begin{equation}\label{eq:DecTriang} \Cone(\varepsilon)[1]\to
     C^\bullet \xrightarrow{\varepsilon^\prime} 
       C^*(C^\bullet) \to 
          \Cone(\varepsilon). \end{equation}       
The standard simplicial splitting of the cylinders $\Delta^i\times\affl$
defines the $\affl$-homotopy between zero and unit sections
  $s_0,s_1\colon C^*(C^\bullet)\to \mathcal Hom(\affl,C^*(C^\bullet))$, 
and hence $C^*(C^\bullet)\in \mathcal A$.
%
On other side $\Cone(\varepsilon)\in \mathcal B$, since it is isomorphic to 
$$Tot(
  \cdots\to
  \mathcal Hom_{\PreTrs}(
     \bZ_{tr}(\Delta^i)/ \bZ_{tr}(pt),
     C^\bullet
   )
  \to\cdots\to
  0 
),$$
and presheaves $ \bZ_{tr}(\Delta^n)/ \bZ_{tr}(pt)$ are $\affl$-contractable,
because of linear homotopy of affine  simplexes 
$
\affl\times\Delta^n\rightarrow \Delta^n\colon 
(\lambda,x_0,x_1,\dots,x_n)\mapsto
(x_0+\lambda\dot \sum_i x_i, 
(1-\lambda)\dot x_1,\dots,(1-\lambda)\dot x_n
)
.$

3b)
Since the third term in the triangle \eqref{eq:DecTriang}
is $C^*(C^\bullet)$, it follows that $C^*$ defined the left adjoint to the embedding.

2b)
It follows form 
triangle \eqref{eq:DecTriang}
that for any $B^\bullet\in \mathcal B$
the morphism $\Cone(\varepsilon(B^\bullet)[-1]\to B^\bullet$
is quasi-isomorphism, and as was mentioned above 
$\Cone(\varepsilon(B^\bullet)$ is a complex of contractable presheaves.
Hence $B^\bullet$ is quasi-isomorphic to a complex with contractable terms.

2c)
%
Let $\mathcal B^\prime$ be the thick subcategory in $\mathcal C$ generated by presheaves
 $\bZ_{tr}(\affl\times X)/\bZ_{tr}(X)$. 
Applying the standard construction 
we get a right adjoint functor to the embedding functor $\mathcal B^\prime\subset \mathcal C$.
Namely, 
consider the natural sequence for any presheaf $ F\in \PreTrs$
\begin{align*}
\xymatrix{
\cdots\ar[r]&
\bP_i  \ar[r]
  \ar[d]^{\varepsilon^i} &
\bP_{i-1}  \ar[r]
  \ar[d]^{\varepsilon^{i-1}} &
\cdots\ar[r] &
\bP_1  \ar[r]
  \ar[d]^{\varepsilon^1} &
\bP_0  \ar[r]
  \ar[d]^{\varepsilon} &
0    
\\
\cdots&
\preF_i\ar@{^(->}[ru] &
\preF_{i-1}\ar@{^(->}[ru] &
\cdots&
\preF_1\ar[ru] &
\preF\ar[ur]
}
,\end{align*}
where \begin{gather*}
\bP_i = 
  \bigoplus\limits_{
    U\in Sm_k, s\in \preF_i(\affl\times U)\colon j^U_0(s)=0
    }
  \bZ_{tr}(\affl\times U)/\bZ_{tr}(U),
\\
{\varepsilon^i}_{(U,s)}\colon
  \bZ_{tr}(\affl\times U)/\bZ_{tr}(U)\simeq \coker_{\PreTrs}(j_0)\xrightarrow{s}\preF_i 
\end{gather*}
$j_0^U\colon 0\times U\hookrightarrow\affl\times U$,
$j_1^U\colon 1\times U\hookrightarrow\affl\times U$,
$\preF_i=\ker(\varepsilon^i)$,
$\preF_0 = \preF$.

If $\preF\in \mathcal B$, then $C^*(\preF)$ is acyclic,
and in particular, $ h^0(C^*(\preF)) = 0$. 
Hence $h^0(C^*(\preF)) = \coker(\varepsilon)$,
and $\varepsilon$ is surjective.
Since 
  $\bZtr(\affl\times X)/\bZtr(X)\in \Dcontr$
  and
  $\varepsilon$ is surjective,
$\preF_1 = \Cone(\varepsilon)[1]\in \mathcal B$.  
Then by induction we get that
for all integer $i$ the homomorphisms $\varepsilon_i$ are surjective
and $\preF_i\in\caffDbaP$.
Thus first row of the diagram above gives resolvent of $\preF$. 

So
any contractable $\preF$ is naturally quasi-isomorphic to a complex in $\mathcal B^\prime$.
Hence via totalization we get the same for any complex $B^\bullet\in \mathcal B$.
%
%
\end{proof}
  
\begin{theorem}\label{th:affsodDSNtr}
Suppose a ringoid 
$\Cor$ satisfies the strictly homotopy invariance axiom and condition \eqref{eq:CorvNisLoc};
then
there is
a semi-orthogonal decomposition 
  $$\DbaSs = \langle\iaffDbaSh,\caffDbaSh\rangle$$
such that
\begin{itemize}[leftmargin=15pt]
\item[1)]
the subcategory $\iaffDbaSh$ is the full subcategory 
spanned by complexes with homotopy invariant sheaf cohomologies, 
and it is the thick subcategory generated by homotopy invariant sheaves;
\item[2)]
the subcategory $\caffDbaSh$ is 
the full subcategory of  $\DbaSs$ 
spanned by complexes Nisnevich quasi-isomorphic to 
    complexes of $\affl$-contractable presheaves,
and 
it is 
the thick subcategory generated 
by sheaves $\bZtrNis(\affl\times X)/\bZtrNis(X)$ for all $X\in Sm_k$; 
\item[3)]
the functor $C^*$ (def. \ref{def:affsimpl}) is exact in respect to the Nisnevich localisation and induces
the left adjoint functor to the embedding $\iaffDbaSh\to \DbaSs$. 
\end{itemize}
%
    
\end{theorem}
\begin{proof}

Let $\caffDbaSh=\mathcal B\subset \DbaSs$ and $\iaffDbaSh=\mathcal A\subset \DbaSs$ be the images of $\caffDbaP$ and $\iaffDbaP$ under the Nisnevich localisation. 
The theorem \ref{affsodDPW} and strictly homotopy invariance axiom implies that 
$$Hom_{\DbaSs}(\bZtrNis(\affl\times X)/\bZtrNis(X), F_{nis}[i])=0, i\in \mathbb Z,$$ for $X\in Sm_k$, and homotopy invariant $F\in \PreTrs,$
and consequently we have $$Hom_{\DbaSs}(B^\bullet, A^\bullet)=0, B^\bullet\in \mathcal B, A^\bullet \in \mathcal A.$$
This yields the semi-orthogonal decomposition $\DbaSs=\langle\mathcal B,\mathcal A \rangle$
and that $C^*$ defies the left adjoint to the embedding $\mathcal A\to \DbaSs$.

%

Now, since  the localisation functor $\DbaPs\to \DbaSs$
takes a presheaf $F\in \PreTr$ to the sheafification $F_{nis}$ (see corollary \ref{cor:NisCohTr}), 
since by the strictly homotopy invariance axiom the sheafification of a homotopy invariant presheaf is homotopy invariant, and since the sheafification of $\bZ_{tr}(X\times\affl)/\bZ_{tr}(X)$ is $\bZtrNis(X\times\affl)/\bZtrNis(X)$, the points 1 and 2 follows. 
%
%
%
\end{proof}

Finally, 
since
$GWCor_k$ and $WCor_k$ over an infinite prefect field $k$
satisfy strictly homotopy invariance axiom
by theorem 
\ref{th:StrictHGWCor},
we can deduce the  
the main result of the article. 
Indeed this is a reformulation of the previous theorem, 
and so the following theorem is true for any $\Cor$ satisfying strictly homotopy invariance axiom and condition \eqref{eq:CorvNisLoc}.

\begin{theorem}\label{th:DMeff}
Suppose the base filed $k$ is prefect and $\chark k\neq 2$; then
\begin{itemize}[leftmargin =12pt]
\item[1)] 
the \emph{category of effective GW-motives} $\DMeffmGW$ over $k$
is equivalent to the full subcategory in the bounded above derived category $\DbaShNGW$ of the category of Nisnevich sheaves with GW-transfers spanned by \emph{motivic complexes}, i.e. complexes $A^\bullet\in \DbaShNGW$ with homotopy invariant sheaf cohomologies $\underline h^i(A^\bullet)$.
\item[2)]
Under the identification from the previous point for any $X\in Sm_k$ GW-motive $M^{GW}_{eff}(X)$ is naturally isomorphic to the complex 
\begin{multline*}
M^{GW}_{eff}(X) = \mathcal Hom 
(\Delta^\bullet, \bZtrNis(X) )=\\ 
[\cdots \to {GWCor}_{nis}(-\times\Delta^i,X)\to 
\cdots \to {GWCor}_{nis}(-,X)]
.\end{multline*}
\item[3)]
there is natural isomorphism
\begin{equation*}\label{eq:}
Hom_{\DMeffmGW}(M^{GW}_{eff}(X),F[i]) \simeq H^i_{nis}(X,F)
\end{equation*}
for all $i\geq 0$,
$X\in Sm_k$, and 
homotopy invariant sheaf with GW-transfers $F$.
\end{itemize}

And similarly for $Witt$-motives.

\end{theorem}
\begin{proof}
1, 2)
It follows from theorem \ref{th:affsodDSNtr} that $\DMeffmGW\simeq \Daff({\ShNGW})$ and there is 
a reflection $\DbaShNGW\colon l_{\aff} \dashv i_{\aff}\colon \DMeffmGW$, 
where 
the reflector $l_{\aff}$ is equivalent to the localisation functor in respect to morphisms of the form $X\times\affl\to X$, 
right adjoint $i_{\aff}$ is equivalent to the embedding of the full subcategory spanned by motivic complexes,
and 
the composition $i_{\aff}(l_{\aff}(-))$ is equal to the functor $C^*=\mathcal Hom_{\DbaShNGW}(\Delta^\bullet, - )$.

3)
The adjunction isomorphism of the pair $_{\aff} \dashv i_{\aff}$
with the isomorphism from theorem \ref{th:ExtTr=Hnis} yields the claim.
%
\end{proof}

\begin{remark}
Theorems \ref{affsodDPW} and \ref{th:affsodDSNtr} or \ref{th:DMeff} give the second square in the diagram \ref{diag:Constr}.
\end{remark}


\begin{thebibliography}{XXXXX}

%
%
%
%
%
%
%
\bibitem{GarkushaPanin_FrMot} G. Garkusha, I. Panin, Framed motives of algebraic varieties (after V. Voevodsky), arXiv:1409.4372
%
%
%
%
%
%
%
%
%
%
%
%
%
%
%



\bibitem{ALP_WittSh-etsinvert}
A.~Ananyevskiy, M.~Levine, I.~Panin,
Witt sheaves and the $\eta$-inverted sphere spectrum,
(Apr 2015), arXiv:1504.04860.

\bibitem{FB_EffSpMotCT}
T.~Bachmann, J.~Fasel,
On the effectivity of spectra representing motivic cohomology theories, arXiv:1710.00594.

\bibitem{Bal_DerWitt}  P.~Balmer, Witt groups Handbook of K-theory, vol. 2, Springer, Berlin (2005), 539-576.

\bibitem{CF_FinChWittCor}
B.~Calmes, J.~Fasel,
Finite Chow-Witt correspondences,
arXiv:1412.2989
\bibitem{ChepInjLocHIiWtrPreSh}
K. Chepurkin,	
Injectivity theorem for homotopy invariant presheaves with W-transfers, 
St. Petersburg Mathematical Journal, 2017, 28:2, 291–297.

\bibitem{AD_DMGWeff}
A.~Druzhinin, Effective Grothendieck-Witt motives and Witt-motives of smooth varieties (expanded version), arXiv:1709.06273.

\bibitem{AD_StrHomInv}
A.~Druzhinin,
Strictly homotopy invariance of Nisnevich sheaves with GW-transfers and Witt-transfers, arXiv:1709.05805.


\bibitem{AD_WtrSh}
A.~Druzhinin, Preservation of homotopy invariance for presheaves with  $\mathrm{Witt}$-transfers under Nisnevich seafication, 
Journal of Mathematical Sciences (2015), Vol. 209, issue 4, pp 555-563 
\bibitem{AD_WittTriangulatedCat}
A.~Druzhinin, Triangulated category of effective Witt-motives $DWM^-_{eff}(k)$,
arXiv:1601.05383 (Jan 2016).



\bibitem{GG_RecRatStMot}
G.~Garkusha,
Reconstructing rational stable motivic homotopy theory, arXiv:1705.01635, 2017. 

\bibitem{GarkushaPanin_FrMot} G. Garkusha, I. Panin, Framed motives of algebraic varieties (after V. Voevodsky), arXiv:1409.4372

\bibitem{GarPan-Kmot14}
G.~Garkusha and I.~Panin,  
The triangulated category of K-motives,
Journal of K-theory, Volume 14, Issue 1 August 2014 , pp. 103-137.

\bibitem{GarPan-Kmot12}
G.~Garkusha and I.~Panin,  
K-motives of algebraic varieties
Homology Homotopy Appl.
Volume 14, Number 2 (2012), 211-264.



%
%

\bibitem{MVW_LectMotCohom}
Mazza, Voevodsky, Weibel, Lecture notes on motivic cohomology, Clay mathematics monographs, ISSN 1539-6061; volume 2.

\bibitem{Suslin-GraysonSpectralSeq}
A.~Suslin,
On the Grayson spectral sequence,
Proceedings of the Steklov Institute of Mathematics. 2003.

\bibitem{SV_Bloch-Kato}
A.~Suslin, V.~Voevodsky,
Bloch-Kato conjecture and motivic cohomology with finite coefficients,
The Arithmetic and Geometry of Algebraic Cycles, 548 (2000), 117--189.

\bibitem{Neem_TriangCat} 
A.~Neeman, Triangulated categories, Annals of Mathematics Studies 148, Princeton University Press (2001).

\bibitem{Krause_localizationinTriangCat}
H.~Krause,
Localization theory for triangulated categories, arXiv:0806.1324v2

\bibitem{Voevodsky_CancellationTh} 
V. Voevodsky, The Cancellation Theorem, Voevodsky, V., "Cancellation theorem", Doc. Math., pp. 671–685, 2010. 


\bibitem{VSF_CTMht_Ctpretr}
V.~Voevodsky, Cohomological theory of presheaves with transfers, 
(in Cycles, Transfers and Motivic Homology Theories. by Eric. M. Friedlander, and Andrei Suslin), 
Annals of Math. Studies, 1999.

\bibitem{VSF_CTMht_DM}
V.~Voevodsky, Triangulated categories of motives over a field, 
(in Cycles, Transfers and Motivic Homology Theories. by Eric. M. Friedlander, and Andrei Suslin), 
 Annals of Math. Studies, 1999.

\bibitem{W_MotComK} 
M. Walker,
Motivic Complexes and the K-Theory of Automorphisms, 
vol. 0205 (1997).

%
%
%
%

\end{thebibliography}
\end{document}